\documentclass[preprint,12pt]{elsarticle}

\usepackage{amsfonts}
\usepackage{amsmath}
\usepackage{amssymb}
\usepackage{amscd}
\usepackage{mathrsfs}  
\usepackage{amsbsy}
\usepackage{amsthm}
\usepackage{graphicx}
\usepackage{fancyhdr}
\usepackage{epsfig}
\usepackage{color}
\usepackage{xy}
\usepackage{enumitem}

\usepackage{CJKutf8}
\usepackage{inputenc}
\usepackage[encapsulated]{CJK}
\usepackage[CJK, overlap]{ruby}

\usepackage[OT2,OT1]{fontenc}  
\newcommand\cyr{%
\renewcommand\rmdefault{wncyr}%
\renewcommand\sfdefault{wncyss}%
\renewcommand\encodingdefault{OT2}%
\normalfont \selectfont} \DeclareTextFontCommand{\textcyr}{\cyr}

\newcommand{\be}{\begin{equation}}
\newcommand{\ee}{\end{equation}}

\newcommand{\bH}{\mathbb{H}}
\newcommand{\K}{\mathbb{K}}
\newcommand{\N}{\mathbb{N}}
\newcommand{\Q}{\mathbb{Q}}
\newcommand{\R}{\mathbb{R}}

\newcommand{\Z}{\mathbb{Z}}

\newcommand{\bS}{\boldsymbol{S}}

\newcommand{\cW}{\mathcal{W}}

\newcommand{\sD}{\mathscr{D}}

\newcommand{\st}{\,:\,}

\newcommand{\sfE}{\mathsf{E}}

\newcommand{\norm}[2]{{\lVert #1\rVert_{#2}}}

\newcommand{\blambda}{\boldsymbol{\lambda}} 
\newcommand{\cF}{\mathcal{F}}

\newcommand{\cC}{\mathcal{C}}

\newcommand{\ff}{\mathfrak{f}}

\newcommand{\ml}{\noindent\vskip 5pt}

\newcommand{\oX}{\overline{X}}

\renewcommand{\rmdefault}{cmr} 
\renewcommand{\sfdefault}{cmr} 
\newtheorem{theorem}{Theorem}
\theoremstyle{plain}

\newtheorem{definition}{Definition}

\newtheorem{examples}{Examples}

\newtheorem{proposition}{Proposition}
\newtheorem{remark}{Remark}

\numberwithin{equation}{section}

\newcommand{\loc}{\mathrm{loc}}

\newcommand{\Supp}{{\textrm{supp\,}}}
\newcommand{\jac}{\mathop{\mathrm{Jac}}}

\newcommand{\GL}{\mathop{\mathrm{GL}}}
\newcommand{\graph}{\mathop{\mathrm{graph}}}

\begin{document}

\begin{frontmatter}

\title{On Local Fractal Functions in\\ Besov and Triebel-Lizorkin Spaces\tnoteref{t1}}
\tnotetext[t1]{Research partially supported by DFG Grant MA 5801/2-1}

\author{Peter R. Massopust\fnref{tum,hmgu}}
\fntext[tum]{Centre of Mathematics, Research Unit M6, Technische Universit\"at M\"unchen, Boltzmannstr. 3, 85747 Garching, Germany, massopust@ma.tum.de}
\fntext[hmgu]{Helmholtz Zentrum M\"unchen,Ingolst\"adter Landstra{\ss}e 1, 8764 Neuherberg, Germany}

\begin{abstract}
Within the new concept of a local iterated function system (local IFS), we consider a class of attractors of such IFSs, namely those that are graphs of functions. These new functions are called local fractal functions and they extend and generalize those that are currently found in the fractal literature. For a class of local fractal functions, we derive explicit conditions for them to be elements of Besov and Triebel--Lizorkin spaces. These two scales of functions spaces play an important role in interpolation theory and for certain ranges of their defining parameters describe many classical function spaces (in the sense of equivalent norms). The conditions we derive provide immediate information about inclusion of local fractal functions in, for instance, Lebesgue, Sobolev, Slodeckij, H\"older, Bessel potential, and local Hardy spaces.
\end{abstract}


\end{frontmatter}

\section{Introduction}\label{sec1}
Iterated function systems, for short IFSs, are a powerful means for describing fractal sets and for modeling or approximating natural objects. IFSs were first introduced in \cite{BD,Hutch} and subsequently investigated by numerous authors. Within the fractal image compression community a generalization of IFSs was proposed in \cite{barnhurd} whose main purpose was to obtain efficient algorithms for image coding. The power of IFSs lies in the fact that they are built around contractive operators acting on complete metric spaces or Banach spaces.

Contractive operators on function spaces are important for the development of both the theory and algorithms for the solution of integral and differential equations. For instance, they are used in the theory of elliptic partial differential equations, Fredholm integral equations of the second kind, Volterra integral equations, and in the theory of ordinary differential equations. For the development of iterative solvers, contractive operators play a fundamental role.

One class of contractive operators is defined on the graphs of functions using a special kind of iterated function system. The fixed point of such an IFS is the graph of a fractal function. There is a vast literature on IFSs and the interested reader is referred to \cite{B0} for a recent review of the topic. Computationally, IFSs are used in areas such as computer graphics to obtain refinement methods that effectively compute points on curves and surfaces \cite{CDM}. They are also employed for the computation of function values of piecewise polynomial functions and wavelets. In fact, it can be shown that these applications use a variant of IFSs where the iterated functions are defined locally \cite{barnhurd}. For more details about these so-called local IFSs and, in particular, their computational applications, we refer to \cite{BHM1}.

In \cite{BHM1}, the generalization of an IFS to a local IFS was reconsidered from the viewpoint of approximation theory and from the standpoint of developing computationally efficient numerical methods based on fractal methodologies. In the current paper, we continue this former exploration of local IFSs and consider two very general classes of functions spaces onto which a certain class of contractive operator acts. In particular, we derive conditions under which such local fractal functions are elements of Besov and Triebel-Lizorkin function spaces. As these two scales of function spaces describe many classical  function spaces such as Lebesgue, Sobolev, Slodeckij, H\"older, Bessel potential, and local Hardy spaces, which are ubiquitous in numerical analysis, the theory of partial differential equations, and harmonic analysis, a deeper understanding of the interplay between contractive operators acting on local fractal functions and these function spaces is warranted.

The outline of this paper is as follows. In Section \ref{sec2}, we introduce the new concept of local IFS and summarize some of its properties. Local fractal functions are then defined in Section \ref{sec3} together with the relevant class of contractive operators. A condition under which the fixed point of such contractive operators belongs to the Banach space of bounded functions is derived. In the next section, we consider the Lebesgue Spaces $L^p$, $0 < p \leq \infty$, as quasi-Banach spaces on $\R^n$, and present a condition so that the fixed points of the contractive operators we investigate belongs to $L^p$. Section \ref{sec5} gives a short introduction to Besov and Triebel--Lizorkin function spaces. The main results are presented in Section \ref{sec6} where explicit conditions are derived for local fractal functions to belong to these two scales of function spaces.

Throughout this paper, we use the following notation. The set of positive integers is denoted by $\mathbb{N} := \{1, 2, 3, \ldots\}$, the set of nonnegative integers by $\N_0 = \N\cup\{0\}$, and the ring of integers by $\Z$. We denote the closure of a set $S$ by $\overline{S}$. (Quasi-)Normed spaces will be denoted by $(X, \norm{\bullet}{X})$ with (quasi-) norm $\norm{\bullet}{\bullet}$. As customary, we also do not distinguish notationally between a space and the set over which it is defined. 

For a vector $\xi := (\xi_1, \ldots, \xi_m)^T\in \R^m$, we define its $p$-quasi-norm by
\be\label{pnorm}
\|\xi\|_p := \begin{cases} 
\displaystyle{\left(\sum_{i=1}^m |\xi_i|^p\right)^{1/p}}, & 0 < p < \infty,\\
\max\{|\xi_i|\st i = 1, \ldots, m\}, & p = \infty.
\end{cases}
\ee
with the usual identification of $\|\bullet\|_\infty$ as $\displaystyle{\lim_{p\to \infty}}\, \|\bullet\|_p$.
\section{Local Iterated Function Systems}\label{sec2}
The concept of \textit{local} iterated function system is a generalization of an IFS as defined in \cite{B,BD,Hutch}. It was first introduced in \cite{barnhurd} and reconsidered in \cite{BHM1}. In what follows, $m\in\N$ always denotes a positive integer and $\N_m := \{1, \ldots, m\}$.

\begin{definition}\label{localIFS}
Suppose that $\{X_i \st i \in \N_m\}$ is a family of nonempty subsets of a Banach space $(X, \norm{\bullet}{X})$. Further assume that for each $X_i$ there exists a continuous mapping $f_i: X_i\to X$, $i\in \N_m$. Then $\cF_{\loc} := \{X; (X_i, f_i)\st i \in \N_m\}$ is called a \textbf{local iterated function system} (local IFS).
\end{definition}

Note that if each $X_i = X$, then Definition \ref{localIFS} coincides with the usual definition of a standard (global) IFS. However, the possibility of choosing the domain for each continuous mapping $f_i$ different from the entire space $X$ adds additional flexibility as will be recognized in the sequel. Also notice that one may choose the same $X_i$ as the domain for different mappings $f\in \cF_\loc$.

\begin{definition}
A local IFS $\cF_{\loc}$ is called \textbf{contractive} if there exists a norm $\norm{\bullet}{*}$ equivalent to $\norm{\bullet}{X}$ with respect to which all functions $f\in \cF_{\loc}$ are contractive (on their respective domains).
\end{definition}
\noindent
We can associate with a local IFS a set-valued operator $\cF_\loc : 2^X \to 2^X$ by setting
\be\label{hutchop}
\cF_\loc(S) := \bigcup_{i=1}^m f_i (S\cap X_i).
\ee

By a slight abuse of notation, we use the same symbol for a local IFS and its associated operator.

\begin{definition}
A subset $A\in 2^X$ is called a \textbf{local attractor} for the local IFS $\{X; (X_i, f_i)\st i \in \N_m\}$ if
\be\label{attr}
A = \cF_\loc (A) = \bigcup_{i=1}^m f_i (A\cap X_i).
\ee
\end{definition}
In \eqref{attr} it is allowed that $A\cap X_i$ is the empty set. Thus, every local IFS has at least one local attractor, namely $A = \emptyset$. However, it may also have many distinct ones. In the latter case, if $A_1$ and $A_2$ are distinct local attractors, then $A_1\cup A_2$ is also a local attractor. Hence, there exists a largest local attractor for $\cF_\loc$, namely the union of all distinct local attractors. We refer to this largest local attractor as {\em the} local attractor of a local IFS $\cF_\loc$.

\begin{remark}
There exists an alternative definition for \eqref{hutchop}. For given functions $f_i$ which are only defined on
$X_i\subset X$, one could introduce set functions (also denoted by $f_i$) that are defined on $2^X$ by setting
\[
f_i (S) := \begin{cases} f_i (S\cap X_i), & S\cap X_i\neq \emptyset;\\ \emptyset, & S\cap X_i = \emptyset,\end{cases}  \qquad i\in \N_m, \; S\in 2^X.
\]
On the right-hand side, $f_i(S\cap X_i)$ is the set of values of the original $f_i$ as in the previous definition. This extension of a given function $f_i$ to sets $S$ which include elements which are not in the domain of $f_i$ basically just ignores these elements. In the following, we use this definition of the set function $f_i$. 
\end{remark}

Now suppose that $X$ is a compact Banach space and the $X_i$, $i\in \N_m$, are closed subsets of $X$, i.e., compact in $X$. If in addition the local IFS $\{X; (X_i, f_i)\st i \in \N_m\}$ is contractive and $f_i(X_i) \subset X_i$, $\forall\,i\in \N_m$,  then the local attractor can be computed as follows. Let $K_0:= X$ and set
\[
K_\ell := \cF_\loc (K_{\ell-1}) = \bigcup_{i\in \N_m} f_i (K_{\ell-1}\cap X_i), \quad m\in \N.
\]
Then 
\[
A_\loc = \lim_{\ell\to\infty} K_\ell, 
\]
where the limit is taken with respect to the Hausdorff metric on $\bH$.
(For a proof, see \cite{barnhurd}.)

In the above setting, a relationship between the local attractor $A_\loc$ of a contractive local IFS $\{X; (X_i, f_i)\st i \in \N_m\}$ and the (global) attractor $A$ of the associated (global) IFS $\{X; f_i\st i \in \N_m\}$ was derived in \cite{mas13}. We state the result without proof and refer the interested reader to \cite{BHM1,mas13}.

\begin{proposition}
Let $X$ be a compact Banach space and let $X_i$, $i\in \N_m$, be closed subsets in $X$. Suppose that the global IFS $\cF:=\{X; f_i\st i \in \N_m\}$ and the associated local IFS $\cF_\loc := \{X; (X_i, f_i)\st i \in \N_m\}$ are both contractive. Then the local attractor $A_\loc$ of $\cF_\loc$ is a subset of the attractor $A$ of $\cF$. 
\end{proposition}

Contractive local IFSs are point-fibered provided $X$ is compact and the subsets $X_i$, $i\in \N_m$, are closed. To show this, define the code space of a local IFS by $\Omega:= \prod_{n\in\N}\N_m$ and endowed it with the product topology $\mathfrak{T}$. It is known that $\Omega$ is metrizable and that $\mathfrak{T}$ is induced by the metric $d_F: \Omega\times\Omega\to \R$,
\[
d_F(\sigma,\tau) := \sum_{n\in \N} \frac{|\sigma_n - \tau_n|}{(m+1)^n},
\]
where $\sigma = (\sigma_1\ldots\sigma_n\ldots)$ and $\tau = (\tau_1\ldots\tau_n\ldots)$. (As a reference, see for instance \cite{Engel}, Theorem 4.2.2.) The elements of $\Omega$ are called codes.

Define a set-valued mapping $\gamma :\Omega \to \K(X)$, where $\K(X)$ denotes the hyperspace of {\em all} compact subsets of $X$, by
\[
\gamma (\sigma) := \bigcap_{n=1}^\infty f_{\sigma_1}\circ \cdots \circ f_{\sigma_n} (X),
\]
where $\sigma = (\sigma_1\ldots\sigma_n\ldots)$. Then $\gamma (\sigma)$ is point-fibered, i.e., a singleton. Moreover, in this case, the local attractor $A_\loc$ equals $\gamma(\Omega)$. For details about point-fibered global IFSs and their attractors, we refer the interested reader to \cite{K}, Chapters 3--5.
\section{Local Fractal Functions}\label{sec3}

Let $X$ be a nonempty connected set and $\{X_i \st i \in\N_m\}$ a family of nonempty subsets of $X$. Suppose $\{u_i : X_i\to X \st i \in \N_m\}$ is a family of bijective mappings with the property that
\begin{enumerate}
\item[(P)] $\{u_i(X_i)\st i \in\N_m\}$ forms a (set-theoretic) partition of $X$: 
\[
X = \bigcup_{i=1}^m u_i(X_i)\;\;\text{and}\;\; u_i(X_i)\cap u_j(X_j) = \emptyset,\; \forall i\neq j\in \N_m.
\]
\end{enumerate}
\noindent
Suppose that $(Y, \norm{\bullet}{Y})$ is a Banach space. Denote by $B(X, Y)$ the set
\[
B(X, Y) := \{f : X\to Y \st \text{$f$ is bounded}\}.
\]
Under the usual definition of addition and scalar multiplication of mappings, and endowed with the norm 
\[
\norm{f - g}{}: = \displaystyle{\sup_{x\in X}} \,\norm{f(x) - g(x)}{Y},
\] 
$(B(X, Y), \norm{\bullet}{})$ becomes a Banach space, the space of bounded functions from $X$ to $Y$. Similarly, $(B(X_i, Y), \norm{\bullet}{})$ is a Banach space

For each $i\in \N_m$, let $\lambda_i \in B(X_i,Y)$ and let $S_i : X_i\to \R$ be a bounded function. For the two $n$-tuples $\blambda := \{\lambda_1, \ldots, \lambda_m\}$ and $\bS := \{S_1, \ldots, S_m\}$, we define an affine operator $B(X,Y)\to Y^X$, called a \textbf{Read--Bajactarevi\'c (RB) operator}, by
\be\label{Phi}
\Phi(\blambda)(\bS) f := \sum_{i=1}^m (\lambda_i\circ u_i^{-1}) \,\chi_{u_i(X_i)} + \sum_{i=1}^m (S_i\circ u_i^{-1})\cdot (f_i\circ u_i^{-1})\,\chi_{u_i(X_i)},
\ee
or, equivalently,
\[
(\Phi(\blambda)(\bS) f)\circ u_i = \lambda_i + S_i\cdot f_i, \quad \text{on $X_i$, $\forall\;i\in\N_m$,}
\]
with $f_i = f\vert_{X_i}$. Here, we explicitly expressed the dependence of $\Phi$ on $\blambda$ and $\bS$. If one or both $n$-tuples of functions are fixed, we suppress the dependence.

As the mappings $\lambda_i$ are bounded in $Y$ and the functions $S_i$ in $\R$, $\Phi(\blambda)(\bS) f$ is also bounded in $Y$. Hence, $\Phi(\blambda)(\bS)$ maps $B(X,Y)$ into itself.

\begin{proposition}
Denote by $\|\bullet\|_{\infty,X_i}$ the $\sup$-norm on $X$ restricted to $X_i$. If $\displaystyle{\max_{i\in \N_m}}\{\|S_i\|_{\infty,X_i}\} < 1$, then $\Phi(\blambda)(\bS)$ is contractive on the Banach space $B(X, Y)$.
\end{proposition}
\begin{proof}
For simplicity, we suppress the dependence of $\Phi$ on $\blambda$ and $\bS$. Let $g,h\in B(X,Y))$ and set $\phi:= g - h$. Note that, since $\{u_i (X_i)\st i \in \N_m\}$ is a partition of $X$ and $\{u_i\st i\in \N_m\}$ is a family of bijections, we have that
\[
\Phi g = \sum_{i=1}^m \Phi g\, \chi_{u_i(X_i)} = \sum_{i=1}^m (\Phi g)\circ u_i \,\chi_{X_i}.
\]
Hence, for all $x' = u_i (x)\in u_i(X_i)$, the following hold.
\begin{align*}
|\Phi f (x') - \Phi g (x')| & = |\Phi\phi (x')| = |\Phi\phi (u_i(x))| = |S_i (x) \cdot \phi_i (x)|  \leq \|S_i\|_{\infty,X_i}\,|\phi_i(x)|\\
& \leq \|S_i\|_{\infty,X_i}\,\|\phi\| \leq \left(\max_{i\in \N_m} \|S_i\|_{\infty,X_i}\right)\|\phi\|,
\end{align*}
where $\phi_i := \phi\vert_{X_i}$. Thus, taking the supremum over all $x'\in u_i(X_i)$ and using the fact again that $\{u_i (X_i)\st i \in \N_m\}$ is a partition of $X$, we obtain
\[
\|\Phi\phi\| \leq \left(\max_{i\in \N_m} \|S_i\|_{\infty,X_i}\right)\|\phi\|,
\]
which proves the claim.
\end{proof}

Therefore, by the Banach Fixed Point Theorem, $\Phi(\blambda)(\bS)$ has a unique fixed point $\ff\in B(X,Y)$, which satisfies the self-referential equation
\be\label{3.4}
\ff = \sum_{i=1}^m (\lambda_i\circ u_i^{-1}) \,\chi_{u_i(X_i)} + \sum_{i=1}^m (S_i\circ u_i^{-1})\cdot (\ff_i\circ u_i^{-1})\,\chi_{u_i(X_i)},
\ee
or, equivalently
\be
\ff\circ u_i = \lambda_i + S_i\cdot \ff_i, \quad \text{on $X_i$, $\quad\forall\;i\in\N_m$}.
\ee
The fixed point $\ff$ is called a \textbf{bounded local fractal function}. When necessary, we denote the dependence of $\ff$ on $\blambda$ and $\bS$ by $\ff(\blambda)(\bS)$.

The following result found in \cite{GHM} and, in more general form in \cite{M97}, is the extension to the setting of local fractal functions:
For fixed $\bS$, the mapping 
\[
\blambda\ni\underset{i=1}{\overset{m}{\times}} B(X_i,Y)\quad\longmapsto\quad\ff(\blambda)\in B(X,Y)
\] 
is a linear isomorphism. For more details and the proof, see \cite{BHM1,mas13}.

It was shown in \cite{BHM1} that $\graph\ff$ is a local attractor of a local IFS naturally associated with the RB--operator $\Phi$. For the sake of completeness, we state this result adapted to our setting without proof.
\begin{theorem}
Consider the family $\cW_\loc := \{X\times  Y; (X_i\times Y, w_i)\st i\in \N_N\}$, where $w_i (x,y):= (u_i (x), \lambda_i (x) + S_i (x) \cdot y)$. Then there exists a norm $\|\bullet\|_\theta$ on $X\times  Y$, so that $\cW_\loc$ is a contractive local IFS and the graph of the local fractal function $\ff$ associated with the operator $\Phi$ given by \eqref{Phi} is an attractor of $\cW_\loc$. Moreover, 
\be\label{GW}
\graph (\Phi\ff) = \cW_\loc (\graph\ff),
\ee
where $\cW_\loc$ denotes the set-valued operator \eqref{hutchop} for the local IFS $\cW_\loc$.
\end{theorem}
\section{Lebesgue Spaces $L^p$, $0 < p \leq \infty$}\label{sec4}
Let $\sfE$ be a real or complex vector space. A mapping $\|\,\cdot\,\|: \sfE\to \mathbb{R}_0^+$ is called a \textbf{quasi-norm} if it satisfies all the usual conditions of a norm except for the triangle inequality, which is replaced by
\be
\|x + y\| \leq c\,(\|x\| + \|y\|)
\ee
for a constant $c\geq 1$. If $c = 1$, then $\|\bullet\|$ is a norm. A complete quasi-normed space is called a \textbf{quasi-Banach space.} 

Recall that the Lebesgue spaces $L^p$, $0 < p \leq \infty$, defined on $\R^n$ are quasi-Banach spaces and, for $1\leq p \leq\infty$, Banach spaces. We also require the following closed subspace of $L^p$. Let $X\subset \R^n$ be a domain, i.e., an open subset of $\R^n$. Define
\[
L^p (\oX) := \{f\in L^p \st \Supp f \subset \oX\}.
\]
As $L^p (\oX)$ inherits its quasi-norm from $L^p$, it is also a quasi-Banach space.

First, we present a result that gives conditions for a fractal function $\ff$ to be an element of the quasi-Banach spaces $L^p (\oX)$. (See, also \cite{BHM1,mas13} for the case $n=1$. There, however, $L^p$, $0 < p < 1$, was considered as a complete metric space.) Here, we also assume that the two $n$-tuples $\blambda$ and $\bS$ are fixed so that we may suppress them in the notation.

\begin{theorem}\label{ThLp}
Let $X$ be a bounded domain in $\R^n$, let $Y:= \R$, and let $0 < p \leq \infty$. In addition to satisfying condition (P), the family of  subsets $\{X_i \st i\in \N_m\} \in 2^X$ and bijective mappings $\{u_i\st i\in \N_m\}$ are supposed to be such that each $u_i$ is a similarity (transformation), i.e., a mapping $X_i \to X$, that enjoys the property
\[
|u_i(x) - u_i(x^\prime)| = \gamma_i |x - x^\prime|, \qquad\forall\; x, x^\prime\in X_i,
\]
for some constant $\gamma_i \in \R^+$. (We do not assume that all $\gamma_i < 1$!) Furthermore, we assume that the functions $\lambda_i \in L^p (\oX)$ and the functions $S_i$ are bounded on $X_i$. Then the RB--operator $\Phi$ defined in \eqref{Phi} maps $L^p(\oX)$ into itself. If, in addition, the condition
\be\label{Lp}
\begin{cases}
\left(\displaystyle{\sum_{i=1}^m}\, \gamma_i^n \,\|S_i\|_{\infty, X_i}^p\right)^{1/p} < 1, & 0 < p \leq \infty;\\
\displaystyle{\max_{i\in\N_m} \left\{\|S_i\|_{\infty,X_i}\right\}} < 1, & p = \infty,
\end{cases}
\ee
holds, then $\Phi$ is also contractive on $L^p (\oX)$.
\end{theorem}
The unique fixed point $\ff:X \subset\R^n\to \R$ of $\Phi$ in $L^p(\oX)$ is called a \textbf{local fractal function of class} $L^p$.

\begin{proof}
Note that under the hypotheses on the functions $\lambda_i$ and $S_i$ as well as the mappings $u_i$, $\Phi f$ is well-defined and an element of $L^p (\oX)$. It remains to be shown that under conditions stated in the theorem, $\Phi$ is contractive on $L^p (\oX)$. 

To this end, suppose that $g,h \in L^p (\oX)$. Let $\phi := g - h$ and denote Lebesgue measure on $\R^n$ by $dm$. Then, for $0 < p <\infty$, we obtain, the following estimates:
\begin{align*}
\|\Phi\phi\|_{L^2}^p & = \int\limits_{\R^n} |\Phi \phi |^p dm = \sum_{i=1}^m \int\limits_{u_i(X_i)} \left| \Phi\phi(x') \right|^p\, dx' = \sum_{i=1}^m \gamma_i^n\,\int\limits_{X_i} \left| \Phi\phi (u_i)(x) \right|^p\, dx\\
& = \sum_{i=1}^m\,\gamma_i^n\,\int\limits_{X_i} \left| S_i (x) \cdot \phi_i(x)\right|^p\,dx \leq \sum_{i=1}^m\,\gamma_i^n\,\|S_i\|^p_{\infty, X_i}\,\int\limits_{X_i} \left| \phi_i(x))\right|^p\,dx\\
&  \leq \left(\sum_{i=1}^{m}\,\gamma_i^n\,\|S_i\|^p_{\infty, X_i}\right) \|\phi\|^p_{{L^p}}.
\end{align*}

Now let $p= \infty$. Then, for $x'\in u_i(X_i)$, we have that
\begin{align*}
|\Phi \phi (x')| & = |\Phi\phi (u_i(x))| = | S_i(x) \cdot \phi_i (x)| \leq \left(\max_{{i\in\N_N}}\,\|S_i\|_{\infty,X_i}\right)\|\phi\|_{L^\infty},\end{align*}
which implies
\[
\|\Phi \phi\|_{L^\infty} \leq \left(\max_{{i\in\N_N}}\,\|S_i\|_{\infty,X_i}\right)\|\phi\|_{L^\infty}.
\]
These calculations prove the claims.
\end{proof}

The conditions given in \eqref{Lp} can be more succinctly written in the following way. Define an Euclidean vector $\xi := (\xi_1, \ldots, \xi_m)^T\in \R^m$ whose entries are given by
\be\label{xi}
\xi_i = \begin{cases}
\gamma_i^{n/p} \|S_i\|_{\infty,X_i}, & 0 < p < \infty,\\
\|S_i\|_{\infty,X_i}, & p = \infty.
\end{cases}
\ee
Then, we can express the statement of Theorem \ref{ThLp} as follows:
\[
\forall\,0 < p \leq\infty:\;\|\xi\|_p < 1\quad\Longrightarrow\quad \ff\in L^p (\oX)
\]
\begin{remark}
In Theorem \ref{ThLp} it is assumed that the $u_i$, $i\in\N_m$, are similitudes. This will be the setting in Section \ref{sec6}. However, as the proof shows, this assumption can be replaced by the weaker condition
\[
\gamma_i := \sup_{x\in X_i} | \det \jac u_i(x)| < \infty, \quad i\in \N_m,
\]
where $\jac$ denotes the Jacobi-matrix. This condition holds, for instance, in the case when each $u_i$ is an affine mapping of the form $u_i = A_i (\bullet) + b_i$, with $A_i \in \GL (n,\R)$ and $b_i \in \R$, $i\in \N_m$.
\end{remark}
\section{Besov and Triebel--Lizorkin Spaces}\label{sec5}
The theory of Besov and Triebel--Lizorkin spaces is very rich and has numerous applications to partial differential equations and approximation
theory, including finite elements, splines and wavelets. Originally, these spaces were developed to close the gaps in the ladders of smoothness spaces
such as the H\"{o}lder spaces $C^s$, $s\in \mathbb{R}_0^+$, and the classical Sobolev spaces $W^{k,n}$, $k,n\in\mathbb{Z}_0^+$. This
section provides a very rudimentary introduction to these two scales of function spaces and the interested reader is referred to \cite{T,TII,TIII} and the references therein.

Recall that the $M$-th order forward difference operator $\Delta_h^M$, $M\in \N$, of step size $h\in\mathbb{R}^n$ acting on a function $f: \R^n \to \R$ is given by
\be
(\Delta_h^M f)(x) := \sum_{\mu = 0}^M (-1)^{M - \mu} {M\choose \mu}\,f(x + \mu\,h).
\ee
In case $f$ is defined on a bounded domain $X\subset \R^n$, we set
\[
\Delta_h^M f(x; X) := 
\begin{cases}
\Delta_h^M f(x), & \textrm{if $x + \mu h\in X$ for $\mu = 0,1,\ldots, M$;}\\
0, &\textrm{otherwise}.
\end{cases}
\]

In the following, we denote the canonical Euclidean norm in $\R^n$ by $|\bullet|$. In addition, we define for $0< p\leq \infty$,
\[
\sigma_p := \frac{1}{\min\{p,1\}} - 1\geq 0,
\]
and for $0< p < \infty$, $0< q\leq \infty$,
\[
\sigma_{n,p,q} := \frac{n}{\min\{p,q\}}.
\]

\begin{definition}[{\cite[Section 2.5.12]{T}}]\label{Besov}
Let $0 < p,q \leq \infty$ and let $s > \sigma_p$. Suppose $M\in\mathbb{N}$ is such that $M > s \geq M - 1$. Then a function $f\in L^p$ belongs to the
\begin{itemize}[leftmargin=*]
\item \textbf{homogeneous Besov space $\dot{B}^s_{p,q} = \dot{B}^s_{p,q} (\R^n)$} iff
\be\label{Bsemi}
|f|_{\dot{B}^s_{p,q}} := \begin{cases}
    \left(\displaystyle{\int_{\mathbb{R}^n} |h|^{-s q} \|\Delta_h^M f\|_{L^p}^q\,\frac{dh}{|h|^n}}\right)^{\frac{1}{q}} < \infty, & 0 < q < \infty;\\ \\
    \displaystyle{\sup_{0\neq h\in\mathbb{R}^n}}\,|h|^{-s}\,\|\Delta_h^M f\|_{L^p} < \infty, & q = \infty.
    \end{cases}
\ee
\item \textbf{inhomogeneous Besov space} $B^s_{p,q} := B^s_{p,q} (\R^n)$ iff
\be\label{B}
\|f\|_{B^s_{p,q}} := \|f\|_{L^p} +  |f|_{\dot{B}^s_{p,q}} < \infty.
\ee
\end{itemize}
\end{definition}
\noindent
$B^s_{p,q}$ is a Banach space for $1\leq p,q \leq \infty$; otherwise $B^s_{p,q}$ is a quasi-Banach space. 

Note that if $P$ is a polynomial of order $M$, then it is in the kernel of the $M$-th order difference operator and, therefore, $|P|_{\dot{B}^s_{p,q}} = 0$.

\begin{definition}[{\cite[Section 2.5.10]{T}}]\label{TL}
Let $0 < p < \infty$, $0 < q \leq \infty$, and suppose $s > \sigma_{n,p,q}$. If $M\in\mathbb{N}$ is such that $M > s \geq M-1$, then a function $f\in L^p$ is said to belong to the 
\begin{itemize}[leftmargin=*]
\item \textbf{homogeneous Triebel--Lizorkin space} $\dot{F}^s_{p,q}(\R^n)$ iff
\be\label{Fsemi}
|f|_{\dot{F}^s_{p,q}} := \begin{cases}
    \left\|\left(\displaystyle{\int_{\mathbb{R}^n} |h|^{-s q}\,|\Delta_h^M f (\bullet)
    |^q\frac{dh}{|h|^n}}\right)^{\frac{1}{q}}\right\|_{L^p} < \infty, & 0 < q < \infty;\\ \\
    \left\|\displaystyle{\sup_{0\neq h\in\mathbb{R}^n}} |h|^{-s}\,|\Delta_h^M f (\bullet)|\,\right\|_{L^p} < \infty, & q = \infty.
\end{cases}
\ee
\item \textbf{inhomogeneous Triebel--Lizorkin space} $F^s_{p,q} := F^s_{p,q} (\R^n)$ iff
\be
\|f\|_{F^s_{p,q}} := \|f\|_{L^p} + |f|_{\dot{F}^s_{p,q}} < \infty.
\ee
\end{itemize}
\end{definition}
\noindent
$F^s_{p,q}$ is a Banach space for $1\leq p,q \leq \infty$; otherwise a quasi-Banach space. Again, polynomials of order $M$ have vanishing semi-norm $|\bullet |_{\dot{F}^s_{p,q}}$.
\begin{remark}
The homogeneous Besov and Triebel-Lizorkin spaces are more compactly described as the linear spaces of all $f\in \L^p$ such that
\be\label{cbesov}
\left\|\,|\bullet|^{-\frac{n}{q}-s}\,\left\|\Delta_{(\bullet)}^M f (\ast)\right\|_{L^p}\right\|_{L^q} < \infty,
\ee
respectively,
\be\label{ctriebel}
\left\|\,\left\||\ast|^{-\frac{n}{q}-s}\,\Delta_{(\ast)}^M f (\bullet)\right\|_{L^q}\right\|_{L^p} < \infty.
\ee
Here, the $L^p$ norm refers to $(\ast)$ and the $L^q$ norm to $(\bullet)$.
\end{remark}
To show the versatility of Besov and Triebel--Lizorkin spaces, some commonly known function spaces are expressed as special cases of these function spaces.
\begin{description}[leftmargin=*]
\item[H\"{o}lder spaces] For $s > 0$ and $s\notin\mathbb{N}$: $C^s = B^s_{\infty,\infty}$.\ml
\item[Sobolev spaces] For $1 < p < \infty$ and $k\in\mathbb{N}_0$: $W^{k,p} = F_{p,2}^k$ and $W^{k,2} = B^k_{2,2}$.\ml
\item[Slodeckij spaces] For $1 \leq p < \infty$ and $< s \notin\N_0$: $W^{s,p} = B^s_{p,p} = F^s_{p,p} $.\ml
\item[Bessel potential spaces] For $1 < p < \infty$ and $s > 0$: $H^{s,p} = F^s_{p,2}$.
\end{description}
\ml\noindent
Here, equality of function spaces is meant in the sense of equivalent quasi-norms.

Now suppose that $X$ is a domain in $\R^n$. Let $A$ be either $B$ or $F$.
\begin{definition}
Let $0 < p,q \leq \infty$ (with $p< \infty$ for the $F$-spaces) and $s\in \R^+$. Then $A^s_{p,q} (\oX)$ is the closed subspace of $A^s_{p,q}$ given by
\[
A^s_{p,q} (\oX) := \left\{f\in A^s_{p,q}\st \Supp f \subseteq\oX\right\}.
\]
\end{definition}
\noindent
$A^s_{p,q} (\oX)$ inherits its norm from $A^s_{p,q}$ and thus is a quasi-Banach space.

In addition to the Banach space $C(X)$ of $\R$-valued uniformly continuous functions on $X$ and the classical smoothness spaces $C^k (X)$, $k\in\N$, we also require in the following the \textbf{Zygmund spaces} $\cC^s$, $s\in \R^+$, which are defined in the following way:
\begin{definition}[{\cite[Definition 2]{TIII}}]
Let $s\in \R^+$ be written as $s = [s]^- + \{s\}^+$, where $[s]^-\in \N_0$ and $0 < \{s\}^+ \leq 1$. Then
\[
\cC^s := \cC^s (\R^n) := \left\{f\in C(\R^n)\st \|f\|_{\cC^s} < \infty\right\},
\]
where
\[
\|f\|_{\cC^s} := \|f\|_{C^{[s]^-}(\R^n)} + \sum_{|\alpha| = [s]^-} \sup_{0\neq h\in \R^n}\,|h|^{-\{s\}^+}\,\|\Delta_h^2 D^\alpha f\|_{C(\R^n)}.
\]
Here, $D^\alpha$ denotes the ordinary differential operator with multi-index $\alpha\in \N_0^n$. 
\end{definition}
\noindent
Zygmund spaces on domains $X\subseteq\R^n$ are defined in a similar way:
\[
\cC^s(\oX) := \{f \in \cC^s \st \Supp f \subset \oX\}.
\]

It is worthwhile stating that the Zygmund spaces $\cC^s$ coincide with the classical H\"older spaces $C^s$ in case $s\in\Q^+\setminus\N$, and that for $s := k\in \N$, $C^k(\R^n) \subsetneq \cC^k$.

We introduced the Zygmund spaces since the functions in $\cC^s$ are pointwise multipliers for the $B$-- and $F$--scale of function spaces. Recall that a function $g$ is called a \textbf{pointwise multiplier} for $A^s_{p,q}$ if the mapping $f\mapsto g\cdot f$ is a bounded linear operator on $A^s_{p,q}$ \cite[Section 2.8]{T}. 

To this end, we have the following result adapted to our setting. The proofs can be found in \cite[Section 4.2.2]{TII}.

\begin{proposition}
Let $0< p\leq \infty$ ($0< p < \infty$ for the $F$--spaces), $0 < q \leq\infty$, and $s\in\R^+$. Suppose $\varrho > \max\left\{s, n\left(\frac1p - 1\right)_+ -s \right\}$, where $(\bullet)_+ := \max\{\bullet, 0\}$. Then
\[
\|g\cdot f\|_{A^s_{p,q}} \leq c\,\|g\|_{\cC^\varrho}\,\|f\|_{A^s_{p,q}}, \quad\textrm{some $c > 0$},
\]
for all $g\in \cC^\varrho$ and all $f\in A^s_{p,q}$.
\end{proposition}
\section{Local Fractal Functions of Besov and Triebel-Lizorkin Type}\label{sec6}
In this section, we again consider the case where $X \in\R^n$ is a domain and $Y := \R$. We derive  conditions so that a fractal function defined by an RB--operator of the form \eqref{Phi} belongs to $A^s_{p,q} (\oX)$. These conditions correct the results presented in \cite{M05}, and extend and generalize those in \cite{mas1,M97,M05,mas2,mas13}.

To this end, we make the following set of assumptions:
\begin{enumerate}
\item[(A1)]	  $X \in\R^n$ is a bounded domain and $Y := \R$.
\item[(A2)] In addition to satisfying condition (P), the family of  subsets $\{X_i \st i\in\N_m\} \in 2^X$ and bijective mappings $\{u_i\st i\in \N_m\}$ are supposed to be such that each $u_i$ is a similarity (transformation), i.e., a mapping $X_i \to X$, that enjoys the property that 
\[
|u_i(x) - u_i(x^\prime)| = \gamma_i |x - x^\prime|, \qquad\forall\; x, x^\prime\in X_i,
\]
for some constant $\gamma_i \in \R^+$. We do not assume that all $\gamma_i < 1$, $i\in \N_m$. Note that each $u_i$ is then of the form $u_i (\bullet) = \gamma_i \,O_i (\bullet) + \tau_i$, where $O_i\in \textrm{SO}(n)$ and $\tau_i\in \R^n$.
\item[(A3)]	Let $0< p\leq \infty$ ($0< p < \infty$ for the $F$--spaces), $0 < q \leq\infty$, and $s\in\R^+$.
\item[(A4)]	For $i\in \N_m$, the function $\lambda_i: X_i \to \R$ belongs to $A^s_{p,q} (\oX)$. 
\item[(A5)]	For $i\in \N_m$, the function $S_i: X_i\to \R$ belongs to the Zygmund space $\cC^\varrho (\oX)$, where $\varrho > \max\left\{s, n\left(\frac1p - 1\right)_+ -s \right\}$.
\end{enumerate}

Examples of fractal functions $\ff:X\subset\R^n\to \R$ that satisfy assumptions (A1) and (A2) in the global setting of IFSs include the family of affine fractal hypersurfaces in $\R^{n+1}$ constructed in \cite{gh,ghm1,ghm2,hm,m90,mas13a}. These constructions may be extended to the setting of local IFSs by taking {\em subsets} of the simplicial set $X$ that are mapped under similitudes onto a partition of $X$. We leave it to the interested reader to provide the details.

For the remainder of this paper, we assume that the two $n$-tuples $\blambda$ and $\bS$ are fixed. We also suppress them in the notation for $\Phi$ and its fixed point. Moreover, if a function $f$ has support in $\oX\subset\R^n$, we -- if need be -- regard it as defined on all of $\R^n$ by setting it equal to zero off $\Supp f$.
\subsection{Local fractal functions in Besov spaces}
First, we consider the case of Besov spaces and derive explicit conditions for the parameters $\gamma_i$ and the functions $S_i$, $i\in \N_m$ so that the associated local fractal function lies in a Besov space.

\begin{theorem}\label{tb}
Suppose the assumptions {\em(A1)} -- {\em(A5)} hold and that $s > \sigma_p$. Then the affine RB--operator given by \eqref{Phi} maps $B^s_{p,q} (\oX)$ into itself. Define a vector $\eta := (\eta_1, \ldots, \eta_m)^T\in \R^m$ whose components are given by: 
\be\label{etai}
\eta_i := \gamma_i^{\frac{n}{p} -s}\, \|S_i\|_{\infty, X_i},\quad i \in\N_m.
\ee
If
\be\label{con:1}
\max\left\{\|\xi\|_p ,\|\eta\|_q \right\} < 1, \quad 0< p, q \leq \infty,
\ee
then $\Phi$ is a contraction. In this case, the unique fixed point $\ff\in B^s_{p,q}(\oX)$ of $\Phi$ satisfies the self-referential equation
\[
\ff\circ u_i = \lambda_i + S_i\cdot \ff_i, \quad \text{on $X_i$, $\quad\forall\;i\in\N_m$},
\]
and is termed a \textbf{fractal function of class $B^s_{p,q}$}.
\end{theorem}

\begin{proof}
Suppose that $f,g\in B^s_{p,q} (\oX)$ and set $\phi := f - g$. Then, $\Supp \Phi f\subseteq\oX$ and $\|\Phi f\|_{B^s_{p,q}} < \infty$, since all $\lambda_i\in B^s_{p,q} (\oX)$ and the functions $S_i$ are pointwise multipliers in $B^s_{p,q}$. Hence, $\Phi$ maps $B^s_{p,q}(\oX)$ into itself. 

First, we consider the case $q < \infty$. On $u_i(X_i)$, the following holds:
\begin{align*}
\int_{\R^n} |h|^{-sq} & \left(\int_{\R^n} \left|\Delta_h^M (\Phi\phi)\right|^p dm\right)^{q/p}\frac{dh}{|h|^n}\\
& = \int_{u_i(X_i)} |h|^{-sq} \left(\int_{u_i(X_i)} \left|\Delta_h^M (\Phi\phi)(x'; u_i(X_i))\right|^p dx'\right)^{q/p}\frac{dh}{|h|^n},
\end{align*}
since only if $h\in u_i(X_i)$ is there any guarantee that $\Delta_h^M (\Phi\phi)(x'; u_i(X_i))\neq 0$.

Hence, using the fact that $u_i = \gamma_i\,O_i + \tau_i$, with $\gamma_i > 0$, $O_i\in \textrm{SO}(n)$ and $\tau_i\in \R^n$, the last expression above implies that
\begin{align*}
\int_{u_i(X_i)} |h|^{-sq} & \left(\int_{u_i(X_i)} \left|\Delta_h^M (\Phi\phi)(x'; u_i(X_i))\right|^p dx'\right)^{q/p}\frac{dh}{|h|^n}\\
& = \int_{u_i(X_i)} |h|^{-sq} \left(\gamma_i^n\,\int_{X_i} \left|\Delta_{\gamma_i^{-1} O_i^{-1}h}^M (S_i\cdot\phi_i)(x; X_i)\right|^p dx\right)^{q/p}\frac{dh}{|h|^n}\\
& \leq \int_{X_i} \gamma_i^{-sq}\,|h|^{-sq} \gamma_i^{nq/p}\,\|S_i\|_{\infty,X_i}^q\,\left(\int_{X_i} \left|\Delta_{h}^M \phi_i (x; X_i)\right|^p dx\right)^{q/p}\frac{dh}{|h|^n}\\
& \leq \gamma_i^{q(\frac{n}{p} - s)}\,\|S_i\|_{\infty,X_i}^q\, \int_{\R^n} |h|^{-sq} \left(\int_{\R^n} \left|\Delta_h^M \phi\right|^p dm\right)^{q/p}\frac{dh}{|h|^n},
\end{align*}
where we used $x' := u_i (x)$.

Thus, 
\[
|\Phi\phi|_{\dot{B}^s_{p,q}} \leq \left(\sum_{i=1}^m \gamma_i^{q\left(\frac{n}{p} -s\right)}\, \|S_i\|_\infty^q\right)^{1/q}\,|\phi|_{\dot{B}^s_{p,q}}.
\]
For $q = \infty$, we obtain, using similar arguments as above, the following inequality on $u_i(X_i)$, $i\in \N_m$:
\begin{align*}
|h|^{- s p}\,\int_{u_i(X_i)} \left|\Delta_h^M\right.&\left.\Phi\phi (x')\right|^p dx' \leq |h|^{- sp} \gamma_i^n\, \|S_i\|_{\infty, X_i}^p \,\int_{X_i} \left|\Delta_{\gamma_i^{-1} O_i^{-1} h}^M \phi_i((x)\right|^p dx\\
& = |h|^{- sp}\,\gamma_i^{n - sp}\, \|S_i\|_{\infty, X_i}^p \,\int_{X_i} \left|\Delta_{h}^M \phi_i((x)\right|^p dx\\
& \leq \left(\max_{i\in\N_m}\left\{\gamma_i^{n/p - s}\, \|S_i\|_{\infty, X_i}\right\}\right)^p\, |h|^{- sp}\,\int_{X_i} \left|\Delta_{h}^M \phi_i((x)\right|^p dx.
\end{align*}
Hence,
\begin{align*}
|h|^{- s p}\,\int_{\R^n} \left|\Delta_h^M\right.&\left.\Phi\phi\right|^p dm = |h|^{-s p} \sum_{i=1}^m \,\int_{u_i(X_i)} \left|\Delta_h^M \Phi\phi (x')\right|^p dx'\\
& \leq \left(\max_{i\in\N_m}\left\{\gamma_i^{n/p - s}\, \|S_i\|_{\infty, X_i}\right\}\right)^p\,|h|^{-sp}\,\sum_{i=1}^m\,\int_{X_i} \left|\Delta_{h}^M \phi_i((x)\right|^p dx\\
& \leq \left(\max_{i\in\N_m}\left\{\gamma_i^{n/p - s}\, \|S_i\|_{\infty, X_i}\right\}\right)^p\,|h|^{-sp}\,\|\phi\|_{L^p}^p,
\end{align*}
implying that
\[
|\Phi\phi|_{\dot{B}^s_{p,q}} \leq \left(\max_{i\in\N_m}\left\{\gamma_i^{n/p - s}\, \|S_i\|_{\infty, X_i}\right\}\right)|\phi|_{\dot{B}^s_{p,q}}.
\]
Therefore, defining a vector $\eta\in \R^m$ whose components are given by \eqref{etai}, applying the $p$-quasinorm introduced in \eqref{pnorm}, and combining the result with that of Theorem \ref{ThLp}, yields
\[
\forall\,0< p,q\leq \infty,\;\forall s > \sigma_p: \quad\max\left\{\|\xi\|_p ,\|\eta\|_q \right\} < 1\quad\Longrightarrow\quad \ff\in B^s_{p,q} (\oX).
\]
%
%
This proves the theorem.
\end{proof}
\begin{examples}\hfill
\begin{enumerate}
\item[{\em (a)}] In the special case that was considered in \cite{M97}, namely the setting of global fractal functions in $\R$, one obtains from \eqref{xi} and \eqref{etai} for $n = 1$, $\gamma_i := 1/m$, $S_i := s_i\in \R$, $i\in \N_m$, $s := k\in\mathbb{N}_0$, and $p = q := 2$, the
condition given in \cite{M97}, namely
\[
\sum_{i=1}^m |s_i|^2 m^{2 k - 1} < 1\quad\Longrightarrow\quad\mathfrak{f} \in W^{k,2}.
\]
Note, that in this case $\xi_i \leq \eta_i$, for all $i\in\N_m$.
\item[{\em (b)}] For the Slodeckij spaces $W^{s,p} = B^s_{p,p}$, $n:=1$, $1 < p <\infty$, and $0 < s \notin\N_0$, we immediately obtain the following condition:
\[
\sum_{i=1}^m |s_i|^p m^{p s - 1} < 1\quad\Longrightarrow\quad\mathfrak{f} \in W^{s,p},
\]
if we set as above $\gamma_i := 1/m$ and $S_i := s_i\in \R$, $i\in \N_m$.
\item[{\em (c)}] For $p = q = \infty$, one obtains requirement for a fractal function to be in $\cC^s$, namely,
\[
\max\{\max_{i\in \N_m} \{\|S_i\|_{\infty, X_i}\}, \max_{i\in \N_m} \{\gamma_i^{-s} \|S_i\|_{\infty, X_i}\}\} < 1.
\]
For the special case of homogenous H\"older spaces $\dot{C}^s$, the above formula extends the one presented in \cite{mas13} for a particular one-dimensional setting.
\end{enumerate}
\end{examples}
\subsection{Local fractal functions in Triebel-Lizorkin spaces}
Here, we state explicit conditions for the parameters $\gamma_i$ and the functions $S_i$, $i\in \N_m$ so that the associated local fractal function lies in a Triebel-Lizorkin space.
\begin{theorem}
Suppose the assumptions {\em(A1)} -- {\em(A5)} hold and that $s > \sigma_{n,p,q}$. Then the affine RB--operator given by \eqref{Phi} maps $F^s_{p,q} (\oX)$ into itself. Let $\eta$ again be the vector with components 
\be\label{zetai}
\eta_i := \gamma_i^{\frac{n}{p} - s}\,\|S_i\|_{\infty, X_i}\quad i \in\N_m.
\ee
If
\be\label{con:2}
\max\left\{\|\xi\|_p ,\|\eta\|_p \right\} < 1, \quad 0< p < \infty,
\ee
then $\Phi$ is a contraction. In this case, the unique fixed point $\ff\in F^s_{p,q}(\oX)$ of $\Phi$ satisfies the self-referential equation
\[
\ff\circ u_i = \lambda_i + S_i\cdot \ff_i, \quad \text{on $X_i$, $\quad\forall\;i\in\N_m$},
\]
and is termed a \textbf{fractal function of class $F^s_{p,q}$}.
\end{theorem}

\begin{remark}
Note that condition \eqref{con:2} is independent of $q$. This independence is a direct consequence of the placement of the norms in the definition \eqref{ctriebel} of a Triebel-Lizorkin space. (Compare this to \eqref{cbesov}!)
\end{remark}

\begin{proof}
Suppose that $f,g\in F^s_{p,q} (\oX)$ and set $\phi := f - g$. Then, $\Supp \Phi f\subseteq\oX$ and $\|\Phi f\|_{F^s_{p,q}} < \infty$, since all $\lambda_i\in F^s_{p,q} (\oX)$ and the functions $S_i$ are pointwise multipliers in $F^s_{p,q}$. Hence, $\Phi$ maps $F^s_{p,q}(\oX)$ into itself. 

We first consider the case $0< q < \infty$. Then, on $u_i(X_i)$, the following hold:
\begin{align*}
\int_{\R^n} \left(\int_{\R^n} |h|^{-sq} \right. & \left. \left|\Delta_h^M \Phi\phi\right|^q \frac{dh}{|h|^n}\right)^{p/q} dm\\
& = \int_{u_i(X_i)} \left(\int_{u_i(X_i)} |h|^{-sq} \left|\Delta_h^M (\Phi\phi)(x'; u_i(X_i))\right|^q \frac{dh}{|h|^n}\right)^{p/q} dx'\\
& = \gamma_i^n\,\int_{X_i} \left(\int_{u_i(X_i)} |h|^{-sq} \left|\Delta_{\gamma_i^{-1}O_i^{-1}h}^M S_i(x)\cdot\phi_i(x;X_i)\right|^q \frac{dh}{|h|^n}\right)^{p/q} dx\\
& \leq \gamma_i^{n-ps}\,\|S_i\|_{\infty, X_i}^p\,\int_{X_i} \left(\int_{X_i} |h|^{-sq} \left|\Delta_h^M \phi_i(x;X_i)\right|^q \frac{dh}{|h|^n}\right)^{p/q} dx\\
& \leq \gamma_i^{n-ps}\,\|S_i\|_{\infty, X_i}^p\,\int_{\R^n} \left(\int_{\R^n} |h|^{-sq} \left|\Delta_h^M \phi\right|^q \frac{dh}{|h|^n}\right)^{p/q} dx.
\end{align*}
This yields
\[
|\Phi\phi|_{\dot{F}^s_{p,q}} \leq \left(\sum_{i=1}^m \,\gamma_i^{p(\frac{n}{p} -s)}\,\|S_i\|_{\infty, X_i}^p\right)^{1/p}\,|\phi|_{\dot{F}^s_{p,q}}.
\]
In case $q = \infty$, we obtain on $u_i(X_i)$:
\begin{align*}
\int_{\R^n} &\left(\sup_{0\neq h\in\mathbb{R}^n} |h|^{-s}\,|\Delta_h^M (\Phi\phi)|\,\right)^p dm\\
& = \int_{u_i(X_i)}\left(\sup_{0\neq h\in u_i(X_i)} |h|^{-s}\,|\Delta_h^M (\Phi\phi)(x'; u_i(X_i))|\,\right)^p dx'\\
& \leq \gamma_i^n\,\|S_i\|^p_{\infty,X_i}\,\int_{X_i}\left(\sup_{0\neq h\in X_i} \gamma_i^{-s}\,|h|^{-s}\,|\Delta_{h}^M \phi_i(x)|\,\right)^p dx\\
& \leq \gamma_i^{n - ps}\,\|S_i\|^p_{\infty,X_i}\,\int_{\R^n}\left(\sup_{0\neq h\in\mathbb{R}^n} |h|^{-s}\,|\Delta_{h}^M \phi(x)|\,\right)^p dx.
\end{align*}
Thus,
\[
|\Phi\phi|_{\dot{F}^s_{p,\infty}} \leq \left(\sum_{i=1}^m \,\gamma_i^{p(\frac{n}{p} -s)}\,\|S_i\|_{\infty, X_i}^p\right)^{1/p}\,|\phi|_{\dot{F}^s_{p,\infty}}.
\]
Therefore, taking the vector $\eta\in \R^m$ whose components are given by \eqref{zetai}, applying the $p$-quasinorm introduced in \eqref{pnorm}, and combining the result with that of Theorem \ref{ThLp}, gives
\begin{align*}
\forall\,0< p < \infty,&\;,\forall\, 0 <q\leq \infty,\,\forall s > \sigma_{n,p,q}:\\
& \;\max\left\{\|\xi\|_p ,\|\eta\|_p \right\} < 1\quad\Longrightarrow\quad\ff\in F^s_{p,q} (\oX),
\end{align*}
which yields the statement in the theorem.
\end{proof}

\begin{examples}\hfill
\begin{enumerate}
\item[{\em (a)}] For a local fractal function to belong to a Bessel Potential space $H^{s,p} = F^s_{p,2}$, we obtain in the special case we considered above, where $n:=1$, $\gamma_i := 1/m$, and $S_i := s_i\in \R$, the criterion:
\[
\sum_{i=1}^m N^{ps - 1} |s_i|^p < 1,
\]
for $1 < p \leq 2$ and $s > 1/p$.
\item[{\em (b)}] Another class of function spaces that are Triebel-Lizorkin spaces for a certain range of the indices $p,q,$ and $s$ are the \textbf{local Hardy spaces}. They are, for our purposes, defined as follows. Let $0 < p <\infty$ and let $\sD_0 := \sD_0(\R^n)$ denote the class of all $C^\infty$-functions $\varphi$ with compact support satisfying $\varphi(0) = 1$. Set $\varphi_t(x) := \varphi (tx)$, for $t > 0$ and $x\in \R^n$. Then
\[
h_p := h_p(\R^n) := \left\{f\in L^p \st \left\|\sup_{0 < t <1}\left|(\cF^{-1}\,\varphi_t\, \cF) f\right|\right\|_{L^p} < \infty\right\},
\]
where $\cF$ denotes the Fourier transform. It can be shown that this definition is independent of the test function $\varphi$ in the sense of equivalent quasi-norms. Moreover, $h_p = F^0_{p,2}$. 

Using the set-up in Example {\em(a)} above, we see that a fractal function belongs to the Hardy space $h_p$ provided that
\[
\sum_{i=1}^m |s_i|^p < N,
\]
for $0 < p < \infty$.
\end{enumerate}
\end{examples}


\begin{thebibliography}{10} 


\bibitem{B} {\sc M.~F.~Barnsley}, {\em Fractals Everywhere}, Dover Publications, New York, 2012.

\bibitem{barninterp} {\sc M.~F.~Barnsley}, {\em Fractal functions and interpolation}, Constr. Approx., 2 (1986), pp.~303--329.

\bibitem{BD} {\sc M.~F.~Barnsley and S.~Demko}, {\em Iterated function systems and the global construction of fractals}, Proc. R. Soc. Lond. A, 399 (1985), pp.~243--275.

%
\bibitem{BHM1} {\sc M.~F.~Barnsely, M.~Hegland and P.~Massopust}, {\em Numerics and fractals}, to appear in Special Issue, Bulletin of the Institute of Mathematics, Academica Sinica (N.S.)

\bibitem{barnhurd} {\sc M.~F.~Barnsley and L.~P.~Hurd}, {\em Fractal Image Compression}, AK Peters Ltd., Wellesly, Massachusetts, 1993.

\bibitem{B0} {\sc M.~F.~Barnsley and A.~Vince}, {\em Developments of Fractal Geometry}, Bull. Math. Sci., 3 (2013), pp.~299-348.


%
%

\bibitem{CDM} {\sc A.~S.~Cavaretta, W.~Dahmen and C.~A.~Micchelli}, {\em Stationary Subdivision}, Mem. Amer. Math. Soc., Vol. 93, No. 453, Providence, R.I., 1991.

\bibitem{Engel} {\sc R.~Engelking}, {\em General Topology}, Helderman Verlag, Berlin, Germany, 1989.
%

\bibitem{gh}    {\sc J.~Geronimo and D.~Hardin}, {\em Fractal interpolation surfaces and a related 2-D multiresolution analysis}, J. Math. Anal. and Appl., 176(2) (1993), 561--586.

\bibitem{ghm1}  {\sc J.~Geronimo, D.~Hardin and P.~R.~Massopust}, {\em Fractal surfaces, multiresolution analyses and wavelet transforms}, in  {\it Shape in Picture} (Y. O, A. Toet, D. Foster, H. Heijmans, and P. Meer, eds.), 275--290, NATO ASI Series, Vol. 126, 1994.

\bibitem{ghm2}  {\sc J.~Geronimo, D.~Hardin and P.~R.~Massopust}, {\em An application of Coxeter groups to the construction of wavelet bases in $\mathbb{R}^n$}, in  {\it Fourier Analysis: Analytic and Geometric Aspects} (W. Bray, P.Milojevi\'{c}, and \v{C}. Stanojevi\'{c}, Eds.), 187--195, Lecture Notes in Pure and Applied Mathematics, Vol. 157, Marcel Dekker, New York 1994.

\bibitem{GHM} {\sc J.~Geronimo, D.~Hardin and P.~Massopust}, {\em Fractal functions and wavelets expansions based on several scaling functions}, J. Approx. Th., {78}(3) (1994), pp.~373--401. 
 
\bibitem{hm}    {\sc D.~Hardin and P.~Massopust}, {\em Fractal interpolation functions from $\mathbb{R}^n$ to $\mathbb{R}^m$ and their projections}, Zeitschrift f\"{u}r Analysis u. i. Anw., {12} (1993), 535--548.

\bibitem{Hutch} {\sc J.~E.~Hutchinson}, {\em Fractals and self similarity}, Indiana Univ. J. Math., 30 (1981), pp.~713--747.

\bibitem{K} {\sc B.~Kieninger}, {\em Iterated Function Systems on Compact Hausdorff Spaces}, Ph.D. Thesis, Augsburg
University, Berichte aus der Mathematik, Shaker-Verlag, Aachen 2002.

\bibitem {lesniak} {\sc K.~Le\'{s}niak}, {\em Stability and invariance of multivalued iterated function systems}, {\em Math. Slovaca}, {53}(2003), pp.~393-405.

\bibitem{m90}   {\sc P.~R.~Massopust}, {\em Fractal surfaces}, J. Math. Anal. and Appl., {151}(1) (1990), 275--290.

\bibitem{mas1} {\sc P.~R.~Massopust}, {\em Fractal Functions, Fractal Surfaces, and Wavelets}, Academic Press, San Diego, 1994.

\bibitem{M97} {\sc P.~R.~Massopust}, {\em Fractal functions and their applications}, Chaos, Solitons, and Fractals, 8(2) (1997), 171--190.

\bibitem{M05} {\sc P.~R.~Massopust}, {\em Fractal functions, splines, and Besov and Triebel-Lizorkin spaces}, in Fractals in Engineering: New trends and applications (J. L\'evy-V\'ehel , E. Lutton, eds.), 21--32, Springer Verlag, London, 2005.

\bibitem{mas2} {\sc P.~R.~Massopust}, {\em Interpolation with Splines and Fractals}, Oxford University Press, New York, 2012.

\bibitem{mas13} {\sc P.~R.~Massopust}, {\em Local fractal functions and function spaces}, to appear in Springer Proceedings in Mathematics \& Statistics.

\bibitem{mas13a} {\sc P.~R.~Massopust}, {\em Fractal hypersurfaces, wavelet sets, and affine Weyl groups}, http://arxiv.org/abs/1309.0241.

%
%

\bibitem{T}     {\sc H.~Triebel}, {\em Theory of Function Spaces}, Birkh\"{a}uer, Basel, 1983.

\bibitem{TII}     {\sc H.~Triebel}, {\em Theory of Function Spaces II}, Birkh\"{a}uer, Basel, 1992.

\bibitem{Tfractal}	{\sc H.~Triebel}, {\em Fractals and Spectra}, Birkh\"{a}uer, Basel, 1997.

\bibitem{Tpaper}	{\sc H.~Triebel}, {\em Function spaces in Lipschitz domains and on Lipschitz manifolds. Characteristic functions as pointwise multipliers}, Revista Mathem\'atica Complutense, 15(2) (2002), 475--524.

\bibitem{TIII}	{\sc H.~Triebel}, {\em Theory of Function Spaces III}, Birkh\"{a}uer, Basel, 2006.

\end{thebibliography}
\end{document}